\documentclass[11pt]{article}
\usepackage{amssymb,amsthm,amsmath}
\newtheorem{thm}{Theorem}[section]
\newtheorem{cor}[thm]{Corollary}
\newtheorem{lem}[thm]{Lemma}

\newtheorem*{thm*}{Theorem}

\theoremstyle{definition}
\newtheorem{defn}[thm]{Definition}

\theoremstyle{remark}
\newtheorem{rem}[thm]{Remark}
\newcommand{\zizj}{z_iz_j^{-1}}
\newcommand{\Ps}{Poincar\'e series}
\begin{document}
\title{Using Hook Schur Functions to Compute Matrix Cocharacters}%submitted to Israel Journal, 5/14/09, to IJAC 6/4/09 pdf to IJAC 7/30
\author{Allan Berele\thanks{Work supported the National Security
Agency, under Grant H98230-08-1-0026}
\\Department of Mathematics\\DePaul University\\Chicago,
IL 60614\\{\tt aberele@condor.depaul.edu}}\maketitle
\begin{abstract} We develop a new integration method based on hook Schur functions instead of Schur functions to compute the cocharacters of matrices.  We then use this method to compute some of the multiplicities in the cocharacter sequence of $3\times 3$ matrices.\end{abstract}
Let $m_\lambda$ be the multiplicity of the irreducible $S_n$-character
$\chi^\lambda$ in the pure trace cocharacter of $k\times k$
matrices over the characteristic zero field~$F$.  These multiplicities have been the subject of too many papers for us to attempt a list.  Recent ones include \cite{B08}, \cite{DD}, \cite{DGV} and~\cite{RV} .  An important
tool in the study of the $m_\lambda$ is the Poincar\'e series
$$P_n(x_1,\ldots,x_n)=\sum m_\lambda S_\lambda(x_1,\ldots,x_n),$$
where $S_\lambda$ is the Schur function.  The Poincar\'e series
can be computed using complex integrals.  It equals $(2\pi
i)^{-k}(k!)^{-1}$ times $$\oint_T
\prod_{i,j=1}^k\prod_{\alpha=1}^n (1-z_iz_j^{-1} x_\alpha)^{-1}
\prod_{1\le i\ne j\le
k}(1-\zizj)^{-1}\frac{dz_1}{z_1}\wedge\cdots\wedge\frac{dz_k}{z_k},$$
where $T$ is the torus $|z_i|=1$, $i=1,\ldots,k$, and the
$x_\alpha$ are taken to be less than~1 in absolute value.  If
instead of Schur functions we work with hook Schur functions we
get the double \Ps $$P_{n,m}(t_1,\ldots,t_n;u_1,\ldots,u_m)=\sum m_\lambda
HS_\lambda(t_1,\ldots,t_n;u_1,\ldots,u_m).$$ Our main result is that
this function can also be computed as a complex integral, namely, $(2\pi i)^{-k}(k!)^{-1}$ times
$$\oint_T \prod_{i,j=1}^k\prod_{\alpha=1}^n\prod_{\beta=1}^m
(1-z_iz_j^{-1} t_\alpha)^{-1}(1+\zizj u_\beta) \prod_{1\le i\ne
j\le
k}(1-\zizj)^{-1}\frac{dz_1}{z_1}\wedge\cdots\wedge\frac{dz_k}{z_k}.$$
We will also prove an analogous formula for mixed trace multiplicities, and to use them
to compute some of the $m_\lambda$ for~$k=3$.  In the first section we prove the formula twice,  the first proof being combinatorial and the second algebraic.  Then, in the second section we apply the formula to determine the pure and mixed trace cocharacters $m_\lambda$ and $\bar{m}_\lambda$ for three by three matrices where $\lambda$ has $\lambda_2\le2$.  These results are recorded in Tables~1 and~2.
\section{Proof of the Formula}
\begin{defn} Let $f(z_1,\ldots,z_k)$ be symmetric and let
$\langle f,1\rangle$ denote $(2\pi i)^{-k}(k!)^{-1}$  times the integral
$$\oint_T f\prod_{1\le i\ne j\le
k}(1-\zizj)\frac{dz_1}{z_1}\wedge\cdots\wedge\frac{dz_k}{z_k},$$
where $T$ is as in the introduction.  As a special case of Weyl's
integration theory, see~\cite{FH} remark~26.20, if $f$ is a
character for the general linear group $GL_k(F)$, then $\langle
f,1\rangle$ is the coefficient of the trivial character
in~$f$.  Or, in terms of modules, if $M$ is a $GL_k(F)$-module with character~$f$, then $\langle f,1\rangle$ is the dimension of the fixed points of~$M$.\end{defn} If $\lambda$ is a partition we denoted by
$S_\lambda(\zizj)$ the evaluation of the Schur
function~$S_\lambda$ on the variables $\zizj$, $i,j=1,\ldots,k$.
The following theorems are from~\cite{F}.
\begin{thm} Let $m_\lambda$ and $\bar{m}_\lambda$ be the
multiplicities of $\chi^\lambda$ in the pure trace cocharacter and
mixed trace cocharacter, respectively, of $M_k(F)$.  Then
$$m_\lambda=\langle S_\lambda(\zizj),1\rangle\mbox{ and
}\bar{m}_\lambda=\langle \sum_{i,j=1}^k \zizj
S_\lambda(\zizj),1\rangle.\label{thm:one}$$\end{thm} In the
introduction we defined the Poincar\'e series $P_n$ and $P_{n,m}$
using the multiplicities~$m_\lambda$.  Analogously, we may define
$\bar{P}_n$ and $\bar{P}_{n,m}$ using $\bar{m}_\lambda$. In order
to make the transition from multiplicities to Poincar\'e series we
need Berele and Remmel's generalization of the Cauchy identities
from~\cite{BR}:
\begin{multline}\sum
HS_\lambda(x_1,\ldots,x_a;u_1,\ldots,u_b)HS_\lambda(y_1,\ldots,y_c;v_1,\ldots,
v_d)=\\
\prod(1+x_iv_j)\prod(1+u_iy_j)\prod(1-x_iy_j)^{-1}\prod(1-u_iv_j)^{-1}.
\end{multline}
Combining this formula with Theorem~\ref{thm:one} we get our main theorem.
\begin{thm} $P_{n,m}$ equals $$\langle \prod_{i,j=1}^k\prod_{\alpha=1}^n\prod_{\beta=1}^m
(1-z_iz_j^{-1} t_\alpha)^{-1}(1+\zizj u_\beta),1\rangle$$ and
$\bar{P}_{n,m}$ equals $$\langle \sum\zizj
\prod_{i,j=1}^k\prod_{\alpha=1}^n\prod_{\beta=1}^m (1-z_iz_j^{-1}
t_\alpha)^{-1}(1+\zizj u_\beta),1\rangle.$$\label{thm:two}
\end{thm}
\begin{proof} The two cases are similar. Here is the pure trace case:\begin{eqnarray*}P_{n,m}&=&\sum_\lambda m_\lambda HS_\lambda(t;u)\\ &=&\sum_\lambda\langle S_\lambda(\zizj),1\rangle HS_\lambda(t;u)\\ &=&\langle\sum_\lambda S_\lambda(\zizj)HS_\lambda(t;u),1\rangle\\ &=& \langle \prod (1+\zizj u)\prod(1-\zizj t)^{-1},1\rangle
\end{eqnarray*}
\end{proof}
We conclude this section by sketching an alternate proof using generic matrices
instead of combinatorics.  Let $K$ be the free supercommutative algebra
generated by the degree~0 (commuting) elements $x_{ij}^\alpha$ and the degree~1
(anticommuting) elements $e_{ij}^\alpha$, $i,j=1,2,\ldots,k$, $\alpha=1,2,\ldots$.
We construct the $k\times k$ matrices $X_\alpha=(x_{ij}^\alpha)$ and $E_\alpha=
(e_{ij}^\alpha)$ and let $U^{n,m}$ be the algebra generated by $X_1,\ldots,X_n$ and
$E_1,\ldots,E_m$.  The algebra $U^{n,m}$ is called the magnum of $M_k(F)$.  It
inherits a $\mathbb{Z}_2$-grading from~$K$ and it has a is
 $(n+m)$-fold grading by degree in the generators.   The map $str:U^{n,m}\rightarrow K$ given
by $str(a_{ij})=\sum a_{ii}$ is a supertrace in the sense that if $A$ and $B$ are homogeneous
with respect to the $\mathbb{Z}_2$ grading, then $$str(AB)=(-1)^{\deg(A)\deg(B)}BA.$$
Denote by $STR^{n,m}$ the subalgebra of $K$ generated by the image of $U^{n,m}$ under $str$.  This algebra
has two important properties which we record as  lemmas.  The former is theorem~4.3 of~\cite{B94}.  The latter is yet another variation of a standard lemma first proven in~\cite{D} and in~\cite{B82}, which has been proven about many types of coharacters:  ordinary, proper, with trace, with involution, with grading, etc.
\begin{lem} The general linear group $GL_k(F)$ acts on $K$ in such a way that $g(t_{ij}^\alpha)$
equals the $(i,j)$-entry of $gX_\alpha g^{-1}$ and $g(e_{ij}^\alpha)$ equals the $(i,j)$-entry of
$gE_\alpha g^{-1}$, for all $g\in GL_k(F)$ and $STR^{n,m}$ is the fixed ring $K^{GL_k(F)}$.
\end{lem}
\begin{lem} Let $f(t_1,\ldots,t_n,u_1,\ldots,u_m)$ be the \Ps\ of $STR^{n,m}$ with respect to
its $(n+m)$-fold grading.  Then this series can be written as a series in hook Schur functions
$$\sum_\lambda m_\lambda HS_\lambda(t_1,\ldots,t_n,u_1,\ldots,u_m),$$ where $m_\lambda$ is the coefficient of $\chi^\lambda$ in the pure trace cocharacter of $M_k(F)$.
\end{lem}
The rest of the proof follows from Weyl's integration formula.  Note that the action of $GL_k(F)$ on $STR^{n,m}$ preserves the grading by degree.  Let $D\in GL_k(F)$ be a generic diagonal matrix, $D=diag(z_1,\ldots,z_k)$.  Each $x_{ij}^\alpha$ and each $e_{ij}^\alpha$ in~$K$ is an eigenvector for~$D$ with eigenvalue equal to $\zizj$.  Hence, for each $\vec{i}\in\mathbb{N}^{n+m}$, the trace of $D$ acting on the degree $\vec{i}$ part of  $STR^{n,m}$ equals the coefficient of $t_1^{v_1}\cdots u_m^{v_{n+m}}$ in $$\prod_{i,j=1}^k\prod_{\alpha=1}^n\prod_{\beta=1}^m
(1-z_iz_j^{-1} t_\alpha)^{-1}(1+\zizj u_\beta).$$  It follows from the two previous lemmas that the \Ps\ of the fixed ring is $P_{n,m}$ and then Weyl's integration formula implies that this \Ps\ is gotten by taking the inner product with~1.  A similar argument holds for the mixed trace cocharacter.
\section{Coefficients $m_\lambda$ for $3\times3$ Matrices}
We first do the case of $n=m=1$ for $3\times3$ matrices.  The series
$P_{1,1}(t,u)$ is given by $\frac1{6(2\pi i)^3}$ times the
integral over $|z_i|=1$, $i=1,2,3$ of $$\prod_{i,j=1}^3 (1+\zizj
u)(1-\zizj t)^{-1} \prod_{i\ne
j}(1-\zizj)\frac{dz_1}{z_1}\wedge\frac{dz_2}{z_2}\wedge\frac{dz_3}{z_3}.$$
This can be computed using Cauchy's residue theorem, which we did using Maple.  Although we are not including the program in this paper, the Appendix records the program we used to compute $P_{1,2}$ which can easily be adapted to the $P_{1,1}$ case.  The result of the computation is a fraction
which can be written as
$$P_{1,1}(t,u)=1+\frac{(t+u)N(t,u)}{(1-t)(1-t^2)(1-t^3)}$$ where
\begin{multline}N(t,u)=(1+t-t^3-t^4+t^5)+u(t+t^2+t^3-t^4)+u^2(1+t+t^2+t^3)\\
+u^3(1+2t+4t^2+t^3)+u^4(1+2t+4t^2+3t^3)+u^5(1+2t+2t^2+3t^3)\\
+u^6(t+2t^2+t^3)+u^7(1+t^2)+u^8\end{multline}
\begin{rem} The \Ps\ $P_{n,0}(t_1,\ldots,t_n)$ staisfies the functional equation
$$P_{n,0}(t_1^{-1},\ldots,t_n^{-1})=(-n)^k(t_1\cdots t_n)^9 P_{n,0}(t_1,\ldots,t_n).$$
Unfortunetly, neither $P_{1,2}(t^{-1},u)$ nor $P_{1,2}(t^{-1},u^{-1})$ can be written in the form $\pm t^a u^b P_{1,2}(t,u)$.\end{rem}
In order to use the \Ps\ to compute the $m_\lambda$ we need this fact about hook Schur functions.
\begin{lem} Let $0\ne\lambda\in H(1,1)$, say $\lambda=(a+1,1^b)$.
Then $HS_\lambda(t,u)=(t+u)t^{a}u^b$.\label{lem:2.1}\end{lem}
\begin{proof} Follows from theorem 6.20 of~\cite{BRv} with
$k=\ell=1$.\end{proof}
We now do a special case that will prove important in subsequent computations.
\begin{cor} $m_{(1^a)}=1$ if $a=0,1,3,4,5,6,8,9$ and it equals zero otherwise.\label{cor:1^a}\end{cor}
\begin{proof} Setting $t=0$ in $P_{1,1}(t,u)$ gives $P_{0,1}(u)$ and it equals $$1+u+u^3+u^4+u^5+u^6+u^8+u^9.$$  On the other hand, $P_{0,1}(u)=\sum m_{(1^a)}u^a.$\end{proof}
It follows from Lemma~\ref{lem:2.1} that the coefficient $m_\lambda$ of $HS_\lambda(t,u)$
in $P_{1,1}$ equals the coefficient of $t^{a}u^b$ in
$N(t,u)(1-t)^{-1}(1-t^2)^{-1}(1-t^3)^{-1}$. One could compute this by hand using partial fractions, but we  used the rgf\_expand command in Maple
to perform the computation.
\begin{thm} For each $b=0,1,\ldots,8$, he coefficient $m_{(a+1,1^b)}$ equals $\alpha(b)a^2+\beta(b)a+O(1)$, where $\alpha(b)$ and $\beta(b)$ are given by the table

\medskip

{
\renewcommand{\arraystretch}{1.5}
\begin{tabular}{c|c|c|c|c|c|c|c|c|c}
$b$&0&1&2&3&4&5&6&7&8\\ \hline
$\alpha(b)$&$\frac1{12}$&$\frac16$&$\frac13$&$\frac23$&$\frac23$
&$\frac34$&$\frac23$&$\frac16$&$\frac1{12}$\\ \hline
$\beta(b)$&$\frac23$&$\frac23$&$1$&$\frac{11}6$&
$\frac43$&$\frac43$&$\frac32$&$\frac23$&$\frac12$
\end{tabular}
}
\label{thm:h11pt}
\end{thm}
To get the series $\bar{P}_{1,1}$ we simply add an additional factor of
$\sum\zizj$ to the integrand.  This yields
$$\bar{P}_{1,1}=(t,u)=1+\frac{(t+u)\bar{N}(t,u)}{(1-t)^2(1-t^2)}$$
where
\begin{multline}\bar{N}(t,u)=(2-2t^2+t^3)+u(2+3t-t^2)+u^2(3+5t+2t^2)\\
+u^3(4+7t+6t^2+t^3)+u^4(4+8t+7t^2+3t^3)+u^5(3+7t+5t^2+3t^3)\\
+u^6(2+4t+3t^2+t^3)+u^7(2+t+t^2)+u^8\end{multline} Here are the
analogues of corollary~\ref{cor:1^a} and theorem~\ref{thm:h11pt}.
\begin{cor} If $b=0,1,\ldots,9$ then the respective values of $\bar{m}_{(1^b)}$ are $1,2,2,3,4,4,3,2,2,1$\label{cor:2.5}\end{cor}
\begin{thm} For each $b=0,1,\ldots,9$, he coefficient $\bar{m}_{(a+1,1^b)}$ equals $\alpha(b)a^2+\beta(b)a+O(1)$, where $\alpha(b)$ and $\beta(b)$ are given by the table

\medskip

{
\renewcommand{\arraystretch}{1.5}
\begin{tabular}{c|c|c|c|c|c|c|c|c|c}
$b$&0&1&2&3&4&5&6&7&8\\ \hline
$\alpha(b)$&$\frac1{4}$&$1$&$\frac52$&
$\frac92$&$\frac{11}2$&$\frac92$&$\frac52$&$1$&$\frac1{4}$\\ \hline
$\beta(b)$&$\frac12$&$\frac72$&$\frac{11}2$&$7$&$\frac{13}2$&
$5$&$\frac72$&$\frac52$&$1$
\end{tabular}}
\end{thm}

 We now wish to compute $m_\lambda$ and $\bar{m}_\lambda$ for $\lambda$ in $H(1,2)$ and we need this analogue of lemma~\ref{lem:2.1} for the one-by-two hook.

\begin{lem} If $\lambda=(\lambda_1,2^a,1^b)\in H(1,2)$ is a partition with
$\lambda_1\ge 2$ then
$$HS_\lambda(t;u_1,u_2)=(t+u_1)(t+u_2)t^{n}S_\mu(u_1,u_2)$$
where $n=\lambda_1-2$ and  $\mu=(a+b,b)=(\lambda_1'-1,\lambda_2'-1)$ and so
$\lambda=(n+2,2^{\mu_1-\mu_2},1^{\mu_2})$
.\label{lem:3.1}\end{lem}
Based on Lemma~\ref{lem:3.1} we will be making use of the correspondence $\lambda\leftrightarrow(n,\mu)$ between $H(1,2)$ and $\mathbb{N}\times H(2,0)$.  See Figure~1 for an example.
\begin{figure}\label{fig:1}
\setlength{\unitlength}{.2in}
\begin{center}
\begin{picture}(9,6)
\multiput(0,0)(0,1){2}{\line(1,0){1}}
\multiput(0,2)(0,1){3}{\line(1,0){2}}
\multiput(0,5)(0,1){2}{\line(1,0){9}}
\multiput(0,0)(1,0){2}{\line(0,1){6}}
\put(2,2){\line(0,1){4}}
\multiput(3,5)(1,0){7}{\line(0,1){1}}
\end{picture}\end{center}\caption{$\lambda=(9,2^3,1^2)$, $n=7$ and $\mu=(5,3)$}
\end{figure}
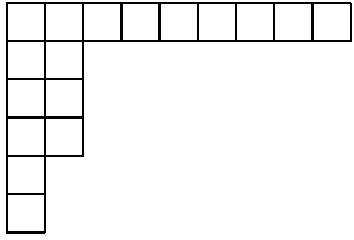

Note that $\lambda_1\le 1$ if and only
if $\lambda\in H(0,1)$ and so in the case of $P_{1,2}$
\begin{align}P_{1,2}(t,u_1,u_2)&=\sum_{\lambda\in H(1,2)} m_\lambda
HS_\lambda(t;u_1,u_2)\notag\\ &= \sum_{\lambda\in H(0,1)}
m_\lambda HS_\lambda(t;u_1,u_2)+\sum_{\lambda\in H(1,2)/H(0,1)}
m_\lambda HS_\lambda(t;u_1,u_2).\label{eq:1}\end{align}  We
already know the multiplicities $m_\lambda$ for $\lambda\in
H(0,1)$ by corollary~\ref{cor:1^a}:  $m_{(1^a)}=1$ for $a=0,1,3,4,5,6,8$ and~$9$ and it is
zero for all other~$a$.  Hence \begin{multline}P_{1,2}^*=_{\rm DEF}\sum_{\lambda\in
H(1,2)/H(0,1)} m_\lambda HS_\lambda(t;u_1,u_2)=\\
P_{1,2}(t,u_1,u_2)-\sum\{HS_{(1^a)}(t,u_1,u_2)|a=0,1,3,4,5,6,8,9\}.\label{eq:5}\end{multline}
Our strategy is to compute $P_{1,2}$ and $P_{1,2}^*$ using theorem~\ref{thm:two} and a Maple
computation which we included in the Appendix for the interested reader.  The results are rational functions with denominator
$$(1-t)(1-t^2)(1-t^3).$$ By lemma~\ref{lem:3.1} $P_{1,2}^*$ can be
re-written as \begin{align*}P_{1,2}^*&=(t+u_1)(t+u_2)\sum_\mu \sum_n m_\lambda t^n
S_\mu(u_1,u_2)\\ &=(t+u_1)(t+u_2)\sum_\mu f_\mu(t) S_\mu(u_1,u_2)\\ &=(t+u_1)(t+u_2)\sum_\mu \frac{g_\mu(t)S_\mu(u_1,u_2)}{(1-t)(1-t^2)(1-t^3)}\end{align*}
where $n=\lambda-2$, $\mu=(\lambda_1'-1,\lambda_2'-1)$ and
 where
$g_\mu(t)$ is a polynomial. We summarize in this lemma.
\begin{lem} Let $\lambda\in H(1,2)$ have $\lambda_1\ge 2$ and let $\mu=(\lambda_1'-1,\lambda_2'-1)$.  Then, for $g_\mu(t)$ as above $m_\lambda$ equals the coefficient of $t^{\lambda_1-2}=t^n$ in $g_\mu(t)(1-t)^{-1}(1-t^2)^{-1}(1-t^3)^{-1}$.\label{lem:2.8}\end{lem}
In order to read off the $g_\mu(t)$ we
divide the numerator of $P_{1,2}^*$ by $(t+u_1)(t+u_2)$ and multiply by
$u_1-u_2$: $g_\mu$ will be the coefficient of
$u_1^{\mu_1+1}u_2^{\mu_2}$. The  result of the computation is in table~1.  In the
interest of saving space, we have omitted the cases of $\mu=(1^a)$
since this case was done in theorem~\ref{thm:h11pt}, and the case
of $\mu_1=8$ because of Formanek's theorem which says that if
$\lambda_9\ge 1$ then
$$m_{(\lambda_1,\ldots,\lambda_9)}=m_{(\lambda_1-1,\ldots,\lambda_9-1)},$$
and so in our case the partitions
$\lambda=(\lambda_1,2^a,1^{8-a})$ and $\tilde\lambda=(\lambda_1-1,1^a)$
have the same multiplicities, $m_\lambda=m_{\tilde\lambda}$.

The case of the mixed trace multiplicities $\bar{m}_\lambda$ is similar.  The integrand has an extra factor of $\sum \zizj$ and the resulting \Ps\ $\bar{P}_{1,2}$ is a rational function with denominator $(1-t)^2(1-t^2)$.  As in equation~(\ref{eq:5}) we define
\begin{align*}
\bar{P}_{1,2}^*=& \sum_{\lambda\in H(1,2)/H(0,1)} m_\lambda HS_\lambda(t;u_1,u_2)\\=& \bar{P}_{1,2}- \sum_{\lambda\in H(0,1)} m_\lambda HS_\lambda (t;u_1,u_2),
\end{align*}
where, by Corollary~\ref{cor:2.5} the last sum is \begin{multline*}HS_0+2HS_{(1)}+2HS_{(1^2)} +3HS_{(1^3)}+4HS_{(1^4)}+4HS_{(1^5)}+3HS_{(1^6)}\\ + 2HS_{(1^7)}+2HS_{(1^8)}+HS_{(1^9)}.\end{multline*}
The result is a function of the form
$$\bar{P}_{1,2}=(t+u_1)(t+u_2)\sum_\mu \frac{\bar{g}_\mu S_\mu(u_1,u_2)}{(1-t)^2(1-t^2)}.$$
Just as in the case of pure trace identities, we can compute the $\bar{g}_\mu$ and use them to compute the $m_\lambda$.  The results are recorded in Table~2.

%in computation on Mac I omitted hs(4) and hs(9)
\newpage
\begin{table}
\begin{tabular}{|r|l|l|}\hline
$\mu$& $g_\mu(t)$&$m_\lambda$ \\ \hline
      %(8,8) & $1$\\ (8,7) &$t^2+1$\\ (8,6)&$t^3+2t^2+t$ \\
 %(8,5) &  $3t^3+2t^2+2t+1$ \\ (8,4)& $3t^3+4t^2+2t+1$\\
  %(8,3) &$t^3+4t^2+2t+1$\\ (8,2)&$t^3+t^2+t+1$ \\(8,1)& $-t^4+t^3+t^2+t$\\
  %(8,0)  & $t^5-t^4-t^3+t+1$\\
  %\hline
  (7,7) &$t^4+2t^2+t+2$&$\frac{n^2+3n}2+O(1)$\\ (7,6) & $t^5+3t^4+3t^3+6t^2+4t+1$&$\frac{3n^2}2+2n+O(1)$ \\
  (7,5) &$3t^5+5t^4+10t^3+9t^2+7t+2$&$3n^2+3n$\\ (7,4)&$3t^5+7t^4+13t^3+13t^2+9t+3$&$4n^2+\frac{9n}2+O(1)$\\
  (7,3) & $t^5+6t^4+9t^3+14t^2+9t+3$&$\frac{7n^2+11n}2+O(1)$\\ (7,2)&
  $2t^4+5t^3+8t^2+6t+3$&$2n^2+\frac{9n}2+O(1)$\\
  (7,1) & $-t^4+2t^3+3t^2+4t+1$&$\frac{3n^2+10n}4+O(1)$ \\ %(7,0)&$t^5-t^4-t^3+2t+1$ \\
  \hline
 (6,6) & $t^6+2t^5+7t^4+5t^3+7t^2+5t+3$&$\frac{5n^2}2+2n+O(1)$\\ (6,5)&
 $3t^6+8t^5+16t^4+19t^3+18t^2+10t+4$&$\frac{13n^2+3n}2+O(1)$\\
  (6,4) & $3t^6+12t^5+22t^4+31t^3+ 26t^2+16t+4$&$\frac{19n^2+5n}2+O(1)$\\ (6,3) & $t^6+8t^5+19t^4+29t^3+27t^2+19t+5$&$9n^2+7n$\\
   (6,2)& $2t^5+8t^4+16t^3+20t^2+14t+6$&$\frac{11n^2}2+9n+O(1)$ \\ (6,1)&
   $t^4+4t^3+8t^2+7t+4$&$2n^2+\frac{11n}2+O(1)$\\
  %(6,0) &$t^2+2t+1$\\
  \hline (5,5)&
  $6t^6+8t^5+18t^4+20t^3+18t^2+8t+6$&$7n^2+O(1)$\\
  (5,4) & $8t^6+19t^5+32t^4+42t^3+37t^2+18t+6$&$\frac{27n^2-n}2+O(1)$\\
  (5,3) & $3t^6+18t^5+30t^4+46t^3+40t^2+25t+6$&$14n^2+\frac{11n}2+O(1)$\\
  (5,2) & $6t^5+15t^4+30t^3+29t^2+20t+8$&$9n^2+11n+O(1)$\\
  (5,1) & $2t^4+9t^3+ 15t^2+ 10t+6$&$\frac{7n^2+17n}2+O(1)$\\
  %(5,0) & $t^5-t^4-t^3+3t^2+3t+3$\\
  \hline
  (4,4) & $6t^6+13t^5+26t^4+ 27t^3+ 27t^2+ 14t+ 7$&$10n^2+n+O(1)$\\
   (4,3) & $ 3t^6+ 16t^5+ 33t^4+ 40t^3+ 40t^2+23t+7$&$\frac{27n^2+11n}2+O(1)$\\
   (4,2) & $6t^5+17t^4+ 32t^3+32t^2 +21t+6$&$\frac{19n^2+21n}2+O(1)$\\
   (4,1) & $2t^4+11t^3+ 16t^2+14t+ 5$&$4n^2+\frac{19n}2$\\
   %(4,0) &$t^5-t^4-t^3+3t^2+ 5t+ 3$\\
    \hline
   (3,3) & $t^6+5t^5+19t^4+18t^3+21t^2+ 13t+7$&$7n^2+6n+O(1)$\\
   (3,2) & $2t^5+11t^4+20t^3+ 24t^2+15t+ 6$&$\frac{13n^2+19n}2+O(1)$\\
   (3,1) & $ 8t^3+12t^2+ 12t+ 4$&$3n^2+8n+O(1)$\\
   %(3,0) & $t^5-t^4-t^3+t^2+5t+3$\\
   \hline
   (2,2) & $3t^4+6t^3+9t^2+6t+6$&$\frac{5n^2}2+6n+O(1)$\\
   (2,1) & $-t^4+3t^3+7t^2+5t+4$&$\frac{3n^2}2+5n+O(1)$ \\
   %(2,0) & $t^5-t^4-t^3+t^2+2t+2$\\
   \hline
   (1,1) & $-t^3+2t^2+2t+3$& $\frac{n^2+5n}2+O(1)$ \\ \hline
   %(1,0) & $-t^3+t^2+t+1$\\
  \end{tabular}\caption{Multiplicities in the Pure Trace
  Cocharacter, $\lambda,\mu$ as in Lemma~\ref{lem:3.1}}
  \end{table}
\newpage
\begin{table}
\begin{tabular}{|r|l|l|}\hline
$\mu$& $g_\mu(t)$&$m_\lambda$ \\ \hline (7,7) &
$t^4+t^3+3t^2+3t+4$ & $3n^2+4n+O(1)$\\ (7,6) &
$t^5+4t^4+7t^3+12t^2+13t+7$ &$11n^2+\frac{9n}2+O(1)$\\ (7,5) &
$3t^5+8t^4+18t^3+24t^2+26t+11$ & $\frac{45n^2+5n}2$\\ (7,4)
&$3t^5+10t^4+23t^3+33t^2+34t+15$ & $\frac{59n^2}2+6n+O(1)$ \\
(7,3) & $t^5+7t^4+16t^3+30t^2+ 32t+16$ &
$\frac{51n^2+31n}2+O(1)$\\ (7,2) & $2t^4+ 6t^3+15t^2+22t+ 13$
&$\frac{29n^2}2+19n+O(1)$\\ (7,1) & $t^3+t^2+11t+8$ &
$\frac{21n^2}4+13n+O(1)$\\ \hline (6,6) & $t^6+3t^5+10t^4+14t^3+
19t^2+16t+9$ & $18n^2-\frac{13n}2+O(1)$\\ (6,5)
&$3t^6+11t^5+27t^4+43t^3+53t^2+44t+19$ & $50n^2-30n+O(1)$\\ (6,4)
& $3t^6+15t^5+37t^4+65t^3+79t^2+70t+27$ & $74n^2-36n+O(1)$
\\ (6,3) & $t^6+9t^5+28t^4+56t^3+75t^2+74t+31$ &
$\frac{137n^2-7n}2+O(1)$ \\ (6,2) &
$2t^5+10t^4+26t^3+44t^2+52t+28$ & $\frac{81n^2}2+28n+O(1)$ \\
(6,1) & $t^4+5t^3+11t^2+23t+18$ & $\frac{29n^2}2+26n+O(1)$\\
\hline (5,5) & $6t^6+14t^5+32t^4+46t^3+56t^2+40t+18$ & $53n^2
-50n+O(1)$\\ (5,4) & $8t^6+27t^5+59t^4+93t^3+111t^2+88t+34$ &
$105n^2-84n+O(1)$
\\ (5,3) & $3t^6+21t^5+51t^4+94t^3+117t^2+108t+42$ &
$109n^2-\frac{79n}2$ \\ (5,2) & $6t^5+21t^4+50t^3+74t^2+83t+40$
&$\frac{137n^2+53n}2$ \\ (5,1) & $3t^4+11t^3+22t^2+38t+28$ &
$\frac{51n^2+77n}2+O(1)$
\\ \hline (4,4) & $6t^6+19t^5+45t^4+66t^3+80t^2+61t+27$ &
$76n^2-61n+O(1)$ \\ (4,3) &
$3t^6+19t^5+52t^4+89t^3+114t^2+100t+43$ & $105n^2-38n+O(1)$
\\ (4,2) & $6t^5+23t^4+54t^3+81t^2+90t+42$ & $74 n^2+28
n+O(1)$ \\ (4,1)& $3t^4+13t^3+24t^2+47t+31$ &
$\frac{59n^2}2+45n+O(1)$\\ \hline (3,3) &
$t^6+6t^5+25t^4+42t^3+59t^2+52t+27$ & $53 n^2-4 n+O(1)$\\ (3,2) &
$2t^5+13t^4+32t^3+55t^2+64t+34$ & $50n^2+34n +O(1)$
\\ (3,1) & $t^4+8t^3+15t^2+39t+27$ & $\frac{45n^2+83n}2+O(1)$
\\ \hline (2,2) & $3t^4+8t^3+18t^2+25t+18$ & $18 n^2+\frac{47
n}2+O(1)$ \\ (2,1) & $2t^3+4t^2+19t+19$ & $11 n^2+\frac{55n}2+
O(1)$\\ \hline (1,1) & $-t^2+4t+9$ & $3 n^2+11 n+O(1)$\\ \hline
\end{tabular}\caption{Multiplicities in the Mixed Trace
  Cocharacter, $\lambda,\mu$ as in Lemma~\ref{lem:3.1}}
\end{table}\clearpage
\section*{Appendix}
\begin{list}{$>$}
\item P:=(n- $>$
product((1+z1/z2*u(i))*(1+z2/z1*u(i))*(1+z1/z3*u(i))\newline
*(1+z3/z1*u(i))*(1+z2/z3*u(i))*(1+z3/z2*u(i))*(1+u(i))\^{}3,i=1..n)):
\item Q:=(n- $>$ product((1-z1/z2*t(i))*(1-z2/z1*t(i))*(1-z1/z3*t(i))\newline *
(1-z3/z1*t(i))*(1-z2/z3*t(i))*(1-z3/z2*t(i))*(1-t(i))\^{}3,i=1..n)):
\item dis:=(1-z1/z2)*(1-z2/z1)*(1-z1/z3)*(1-z3/z1)*(1-z2/z3)*(1-z3/z2)\newline /z1/z2/z3/6:
\item A:=P(2)/Q(1)*dis:\quad\{{\it A is the function we will integrate to compute $P_{1,2}$.  We omitted the factor of $(2\pi i)^{-3}$ because it seems easiest for Maple to compute the integral using the residue command.  To compute $\bar{P}_{1,2}$ we would need to add a factor of } (z1+z2+z3)*(1/z1+1/z2+1/z3)\}.
\item A1:=residue(A,z1=0)+residue(A,z1=t(1)*z2)+residue(A,z1=t(1)*z3):
\item A1:=simplify(A1):
\item factor(denom(A1));\quad \{{\it We need the denominator to compute the next residues.}\}
$$6z2^5z3^5(t(1)+1)(t(1)-1)^3(t(1)z2-z3)(z2-z3t(1)^2)(z2-t(1)z3)(z2-t(1)^2z3)$$
\item A2:=residue(A1,z2=0)+residue(A1,z2=z3*t(1))+residue(A1,z2=z3*t(1)\^{}2):
\item A2=simplify(A2):
\item factor(denom(A2)); $$t(1)^{12}(t(1)+1)^6(t(1)-1)^{21}
z3^{13}(t(1)^2+t(1)+1)^2(t(1)^2+1)$$
\item A3:=residue(A2,z3=0);\quad \{{\it A is $P_{1,2}$}\}
\item h := a-$>$ sum(u(1)\^{}i*u(2)\^{}(a-i), i = 0 .. a);\quad\{{\it $h(n)$ is $h_n(u(1),u(2))$}\}
\item hs :=  n-$>$ h(n)+t(1)*h(n-1);\quad\{{\it hs(n) is $HS_n(t(1),t(2);u(1))$}\}
\item B := simplify(A3-hs(0)-hs(1)-hs(3)-hs(4)-hs(5)-hs(6)-hs(8)-hs(9)); \quad\{{\it B is $P_{1,2}^*$}\}
\item B1 := numer(B)*(u(1)-u(2))/((t(1)+u(1))*(t(1)+u(2)));
\item g:=-((a,b)-$>$coeff(coeff(B1, u(1)\^{}(a+1)),u(2)\^{}b));
\item for a from 1 to 9 do for b from 1 to a do print(g(a,b),a,b); end do end do
\end{list}

\end{document}